\documentclass[leqno,11pt]{amsart}

\usepackage{bbm}
\usepackage[rgb,dvipsnames]{xcolor}
\usepackage{tikz} 
\usetikzlibrary{decorations.text}
\usetikzlibrary{arrows,arrows.meta,hobby}
\usetikzlibrary{shapes.misc, positioning}

\usepackage{tikz-cd}

\usetikzlibrary{cd}

\usepackage{setspace}
\usepackage[english]{babel}
\usepackage{bussproofs}
\usepackage[utf8]{inputenc}
\usepackage{csquotes}
\usepackage{mathtools}
\usepackage{mathrsfs}
\usepackage{latexsym}
\usepackage{enumitem}
\usepackage{amsthm}
\usepackage{amssymb}
\DeclareMathAlphabet{\mathpzc}{OT1}{pzc}{m}{it}

\usepackage{bbm}

\newtheorem{theorem}{Theorem}[section]
\newtheorem*{theorem*}{Theorem}
\newtheorem{maintheorem}{Main Theorem}[section]
\newtheorem{conjecture}{Conjecture}[section]
\newtheorem{lemma}[theorem]{Lemma}
\newtheorem{proposition}[theorem]{Proposition}
\newtheorem{corollary}[theorem]{Corollary}

\newtheorem{fact}[theorem]{Fact}
\newtheorem{claim}[theorem]{Claim}
\theoremstyle{definition}
\newtheorem{definition}[theorem]{Definition}

\theoremstyle{remark}
\newtheorem{remark}{Remark}

\newtheorem{question}{Question}

\def\hook{\upharpoonright}
\def\forces{\Vdash}

\def\ZFC{\mathsf{ZFC}}

\def\PFA{\mathsf{PFA}}
\def\MA{\mathsf{MA}}

\def\BA{\mathsf{BA}}

\def\baire{\omega^\omega}

\def\cantor{2^\omega}

\def\mfb{\mathfrak b}

\def\CH {\mathsf{CH}}
\def\Q{\mathbb Q}
\def\P{\mathbb P}

\def\mfp{\mathfrak{p}}

\title{A note on adding isomorphisms and the pseudointersection number}

\author[Switzer]{Corey Bacal Switzer}
\address[C.~B.~Switzer]{Institut f\"{u}r Mathematik, Kurt G\"odel Research Center, Universit\"{a}t Wien, Kolingasse 14-16, 1090 Wien, AUSTRIA}
\email{corey.bacal.switzer@univie.ac.at}

\thanks{\emph{Acknowledgments:} This research was funded in whole or in part by the Austrian Science Fund (FWF) through the following grant: 10.55776/ESP548.}

\begin{document}

\begin{abstract}
   We prove that for every tower $\mathcal T$ there are $\aleph_1$-dense $A$ and $B$ so that any ``reasonable" forcing notion $\P$ - an adjective that includes all known ones - for making $A$ and $B$ isomorphic will add a pseudointersection for the tower. This shows in particular that $\MA_{\aleph_1}(\sigma{\rm -centered})$ holds in all known models of $\BA$, which provides intrigue to well known questions of Todor\v{c}evi\'c and Stepr\=ans-Watson. 
\end{abstract}

\maketitle

\section{Introduction}

Recall that a set of reals $A\subseteq \mathbb R$ is $\aleph_1${\em -dense} if its intersection with each nonempty open interval has size $\aleph_1$. Alternatively an arbitrary linear order $(L, <)$ is $\aleph_1$-dense just in case each nonempty interval has size $\aleph_1$. Clearly for separable linear orders (without endpoints) these definitions are equivalent. {\em Baumgartner's axiom}, hereafter denoted $\BA$, is the statement that all $\aleph_1$-dense sets of reals are isomorphic. $\BA$ is well studied in set theory, see e.g. \cite{Baum73, weakBA, BaumPFA, ARS85, AvrahamShelah81}. It was proved consistent in \cite{Baum73}. It is well known that $\BA$ can be forced over a model of $\CH$ (\cite{Baum73}), and follows from $\PFA$ (\cite{BaumPFA}) but not $\MA + \neg \CH$ (\cite{AvrahamShelah81}). The most important motivation for the current paper is the following theorem of Todor\v{c}evi\'{c} from \cite{Todorcevic89}. 

\begin{theorem}(\cite{Todorcevic89}) \label{todorcevic}
$\BA$ implies $\mfb > \aleph_1$. 
\end{theorem}

The precise definition of $\mfb$ and other cardinals used in this introduction will be given below in Subsection 1.1 as part of preliminaries.

In this short note we provide some commentary on Todor\v{c}evi\'c's theorem by proving that any ``reasonable" way of forcing $\BA$ will in fact force many more well studied cardinal characteristics to be large.
\begin{maintheorem} \label{introtheorem}
    If $\P$ is an iteration of {\em reasonable} forcing notions which forces $\BA$ then in any generic extension by $\P$ it is necessary that $\mfp > \aleph_1$. 
\end{maintheorem}

The precise definition of ``reasonable" is given in Section 2. We note now that it encompasses the forcing notions used in \cite{Baum73}, \cite{BaumPFA}, \cite{ARS85} and \cite[Chapter 15]{todorcevicnotesonforcingaxioms} for studying $\BA$, which to the best of our knowledge are more or less all of them that appear in the literature. 

This theorem provides not quite a partial answer but some intrigue to questions from \cite{Todorcevic89} and \cite{Stepranswatson87}. To explain this better we recall the following, easy to check proposition, see \cite{Stepranswatson87} for a proof.

\begin{proposition}
    $\BA$ is equivalent to the statement that for all $\aleph_1$-dense $A, B \subseteq \mathbb R$ there is an autohomeomorphism $h:\mathbb R \to \mathbb R$ mapping $A$ to $B$. \label{propforBA}
    \end{proposition}
 This latter formulation is useful for obtaining analogues of $\BA$ on other topological spaces. The following definition along these lines is due to Stepr\=ans-Watson, see \cite{Stepranswatson87}.
 \begin{definition}
     Let $X$ be a topological space and $\kappa$ a cardinal. We say that $\BA_\kappa(X)$ holds just in case for each $A, B \subseteq X$ which are $\kappa$-dense there is an autohomeomorphism $h:X \to X$ mapping $A$ to $B$. 
 \end{definition}
Here $\kappa$-dense means $A \cap U$ has size $\kappa$ for each nonempty, open $U \subseteq X$. Thus the content of Proposition \ref{propforBA} is that $\BA$ is equivalent to $\BA_{\aleph_1}(\mathbb R)$. The case of $\mathbb R$ turns out to be unique along these lines. 
\begin{fact} \label{BAfacts}
    \begin{enumerate}
        \item For all $\kappa < \mfp$ we have $\BA_\kappa (\cantor)$ and $\BA_\kappa (\baire)$ (\cite{BaldwinBeaudoin89}).
        \item For all $\kappa < \mfp$ and all natural numbers $n > 1$ we have $\BA_\kappa (M^n)$ and $\BA_\kappa (\mathbb R^n)$ hold for any compact $n$-dimensional manifold $M^n$ (\cite{Stepranswatson87}).
        \item $\BA := 
        \BA_{\aleph_1} (\mathbb R)$ does not follow from $\MA + \neg \CH$. In particular $\aleph_1 < \mfp$ does not suffice to conclude $\BA$ (\cite{AvrahamShelah81}.
    \end{enumerate}
\end{fact}

In particular, $\BA_{\aleph_1} (\mathbb R^n)$ does not imply $\BA$ for any $n > 1$. The {\em Stepr\=ans-Watson conjecture}, formulated in \cite{Stepranswatson87}, asks whether the converse holds.
\begin{conjecture} \label{conjecture}
  $\BA$ implies $\BA_{\aleph_1}(\mathbb R^n)$ for all $n > 1$. 
\end{conjecture}

By Fact \ref{BAfacts}, item (2), to prove this conjecture positively it would be enough to improve Todor\v{c}evi\'c's Theorem \ref{todorcevic} to concern the pseudointersection number. This was in fact already asked by Todor\v{c}evi\'c, in \cite{Todorcevic89} and more recently discussed in the introduction to \cite{TodorcevicGuzman24}\footnote{In \cite{Todorcevic89} in fact Todor\v{c}evi\'c asks about $\mathfrak{t}$ however by the main result of \cite{p=t}, $\mfp = \mathfrak{t}$, though this was not known at the time.}.
\begin{question} \label{pquestion}
    Does $\BA$ imply $\mfp > \aleph_1$?
\end{question}

Thus Main Theorem \ref{introtheorem} states that were the answer to either Conjecture \ref{conjecture} or Question \ref{pquestion} to be negative a completely new idea for constructing models of $\BA$ would be needed. 

The rest of this paper is organized as follows. In the remainder of this section we review some preliminaries we will need. In Section 2 we introduce ``reasonable" forcing notions and prove forcing $\BA$ with reasonable forcing notions will force $\mfp > \aleph_1$.  In Section 3 we sketch a similar result for forcing $\BA$ for other spaces and in Section 4 we conclude with an observation that $\BA$ implies a formal strengthening that suggests further methods for attacking Question \ref{pquestion}. 

\subsection{Preliminaries}
The rest of this section is devoted to recalling a few facts and definitions we will need. Our notation is mostly standard, conforming to the monographs \cite{JechST} and \cite{KenST}. We also refer the reader to \cite{BlassHB} or \cite{BarJu95} for any undefined notions about cardinal characteristics. To begin, we recall the definitions of the cardinal characteristics we will need.

\begin{definition}
    \begin{enumerate}

        \item Given $f, g \in \baire$ we write $f\leq^*g$ if and only if for all but finitely many $k < \omega$ we have $f(k) \leq g(k)$. In this case we say $g$ {\em eventually dominates} $f$. The {\em bounding number}, denoted $\mfb$ is the least size of an unbounded set $\mathcal B \subseteq \baire$, i.e. a set so that no single $g\in \baire$ eventually dominates every $f \in \mathcal B$.         

        \item If $\mathcal F$ is a family of infinite subsets of natural numbers we say that $\mathcal F$ has the {\em strong finite intersection property} if for every finite subset $\mathcal A \subseteq \mathcal F$ we have that $\bigcap \mathcal A$ is infinite. If $A \subseteq \omega$ is infinite then we say that $A$ is a {\em pseudointersection} of such a family $\mathcal F$ if $A \subseteq^* B$ for all $B \in \mathcal F$ - i.e. for each $B$ there is a $k < \omega$ with $A \setminus k \subseteq B$. In this case we also say $A$ is {\em almost contained in} $B$. The {\em pseudointersection number} $\mfp$ is the least size of a family with the strong finite intersection property but no pseudointersection. 

        \item A {\em tower} $\mathcal T \subseteq [\omega]^\omega$ is a set that is reverse well ordered by $\subseteq^*$. The {\em tower number}, denoted $\mathfrak{t}$ is the least size of a tower with no pseudointersection.
    \end{enumerate}
\end{definition}

Recall that $\mfp \leq \mfb$ provable in $\ZFC$, see e.g. \cite{BlassHB}. We also need the following celebrated result of Shelah and Malliaris from \cite{p=t}. 

\begin{fact}
    It is a $\ZFC$ fact that $\mfp = \mathfrak{t}$.
\end{fact}

Finally most of the work in proving Main Theorem \ref{introtheorem} involves translating between the combinatorics of $2^\omega$ and the topology of $\mathbb R$. As part of this we will often associate elements of $P(\omega)$ with their characteristic functions i.e. treat them as elements of $2^\omega$. In this vein we often abuse notation by saying things like ``for $Y \in 2^\omega$ if $n \in Y$..." which of course really means $Y(n) = 1$. Denote by $[x(n)=1]$ or $[x(n)=0]$ the basic open subsets of $2^\omega$ consisting of all $Y$ so that $Y(n) = 1$ (or $Y(n) = 0$ respectively).  The key tool for translating combinatorics from $\cantor$ to topology on $\mathbb R$ is the Cantor-Lebesgue map. 

\begin{definition}
    The {\em Cantor-Lebesgue} map, $\lambda:2^\omega \to [0, 1]$, is defined by $\lambda (x) = \Sigma_{i \in \omega} \frac{x(n)}{2^{n+1}}$. 
\end{definition}

This is a well studied object, see e.g. \cite{weiss} for a thematically similar application. The following facts are well known and easily verified.
\begin{fact}
    \begin{enumerate}
        \item $\lambda$ is continuous (everywhere), onto and one-to-one when restricted to the set of non eventually constant points.
        \item The set of $d \in [0, 1]$ so that $|\lambda^{-1}(\{d\})| > 1$ is simply the set of diadic rationals and for each such $|\lambda^{-1}(\{d\})| = 2$. In particular it is two-to-one on a countable dense set and one-to-one outside of that.
        \item For each $n \in \omega$ we have $\lambda ``[x(n)=1] = \bigcup_{i\in 2^{n}} [\frac{2i+1}{2^{n+1}}, \frac{2i + 2}{2^{n+1}}]$. 
    \end{enumerate}
\end{fact}

Let us just remark in English what the final item above says as this is key in the proof of Theorem \ref{introtheorem}. The image of the set $[x(n)=1] \subseteq \cantor$ under $\lambda$ is defined as follows: first divide $[0,1]$ into $2^{n+1}$ equal pieces (the endpoints will be diadic rationals), then take the union of the even pieces. For instance, the image of $[x(1) = 1]$ is the set $[\frac{1}{4}, \frac{1}{2}] \cup [\frac{3}{4}, 1]$ i.e. we divide the interval into quarters and take the second and fourth ones. 

\section{Adding Pseudointersections with Reasonable Forcing}

In this section we prove Main Theorem \ref{introtheorem}. Modulo some details the main point is the following which we will show the following.

\begin{theorem} \label{killing towers}
    For every tower $\mathcal T$ of size $\aleph_1$ there are $\aleph_1$-dense sets of reals $A_{\mathcal T}$ and $B_{\mathcal T}$ so that if $\P$ is reasonable for $A_{\mathcal T}$ and $B_{\mathcal T}$ then in any generic extension by $\P$ there is a pseudointersection to $\mathcal T$. \label{LOmain}
\end{theorem}

The definition of reasonable is given below. Main Theorem \ref{introtheorem} follows from this as, given an iteration of reasonable forcing notions which force every pair of $\aleph_1$-dense sets to be isomorphic we would cofinally often have to deal with sets of the form $A_\mathcal T$ and $B_\mathcal T$ as in the theorem above and therefore every tower of size $\aleph_1$ in the extension would have a pseudointersection. Therefore $\mathfrak{t}$ and hence $\mathfrak{p}$ would necessarily need to be larger than $\aleph_1$. Let us now define reasonableness. 

\begin{definition}
    
Let $A$ and $B$ be $\aleph_1$-dense sets of reals. A partial order $\P$ is called {\em reasonable} for $A$ and $B$ if it satisfies the following three conditions.
\begin{enumerate}
\item (Finite Isomorphism) Every condition $p \in \P$ is a finite partial isomorphism from $A$ to $B$. 
    \item (Restriction) If $p \in \P$ and $q = p \hook Z$ for some finite $Z \subseteq {\rm dom}(p)$ then $q \in \P$.
    \item (Dense Mapping Property) Let $p \in \P$ and $x \in A \setminus {\rm dom}(p)$. Suppose $x_0 < x < x_1$ so that $x_0, x_1 \in {\rm dom}(p)$ with $x_0$ the largest such and $x_1$ the smallest such. If $U \subseteq \mathbb R$ is an open interval contained in $(p(x_0), p(x_1))$ then there is a $q \leq p$ so that $x \in {\rm dom}(q)$ and $q(x) \in U$. 
\end{enumerate}

\end{definition}

The dense mapping property is the most important. It states roughly that, within the restrictions given by $p$, for any $x \notin{\rm dom}(p)$ there is a dense set of elements of $B$ into which we can choose to map $x$ by an extension of $p$. By the restriction condition we can take this strengthening to in fact be simply $p$ plus a single additional point - namely $\langle x, q(x)\rangle$, a fact we will use implicitly in the proof of Lemma \ref{mainlemma} below. Note that it is easy to check that if $\P$ is reasonable for $A$ and $B$ then it will generically add a linear order embedding from $A$ to $B$ since, by the dense mapping property, it is dense to add any particular $a \in A$ to the domain. Thus if $G \subseteq \P$ is generic then $\bigcup G$ will in fact be an embedding from $A$ to $B$. The notion of reasonable includes of course very non-proper cases, such as the collection of {\em all finite partial isomorphisms from} $A$ to $B$, but also the forcing notions used to produce $\BA$ in \cite{Baum73}, \cite{BaumPFA}, \cite{ARS85} and \cite[Chapter 15]{todorcevicnotesonforcingaxioms}. We do not in fact know of a way of producing isomorphisms generically between a given $A$ and $B$ which are not reasonable (modulo some kind of preparation such as collapsing to force $\CH$ such as in \cite{BaumPFA}).

\begin{remark}
    In this vein we also mention (for the reader in the know) the forcing of \cite[Theorem 93]{todorcevicnotesonforcingaxioms} used to produce an embedding of $B$ into $X$ for $B$ of size $\aleph_1$ and $X$ which is $\aleph_1$-dense. This forcing notion consists of pairs $p = (f_p, N_p)$ where $f_p$ is a finite partial embedding and $N_p$ is a finite continuous $\in$-chain of elementary submodels used as side conditions (alongside some restrictions indicating how $f_p$ interacts with $N_p$). This forcing is not literally reasonable because of the side condition part but its ``working part" forms a reasonable poset and it is easy to check that the arguments of Lemma \ref{mainlemma} below apply equally well in this case. As it would take us too far afield we leave the details to the interested reader to check. 
\end{remark}

Next we prove the following lemma which allows us to add pseudointersections. Fix a tower of size $\aleph_1$ for the rest of this section, say $\mathcal T = \{X_\alpha \; | \; \alpha < \omega_1\}$ where $\alpha < \beta$ implies $X_\beta \subseteq^* X_\alpha$. Moreover without loss we assume $X = X_0$ is infinite, co-infinite. Finally, again without loss we assume that $\mathcal T$ is $\aleph_1$-dense. 

\begin{lemma}
Let $B \subseteq 2^\omega$ be $\aleph_1$-dense consisting of elements which are infinite and almost disjoint from $X$. Let $A = \{\omega \setminus X_\alpha\mid \alpha < \omega_1\}$. Note that $A$ is also $\aleph_1$-dense. If $f:A \to B$ is a function so that there are infinitely many $n \in X$ for which $f``[x(n) =1] \subseteq \bigcup_{m \in X \setminus n} [x(m)=1]$ then $\mathcal T$ has a pseudointersection. \label{little}
\end{lemma}

We remark that no assumption on $f$ is made besides it being a function. In particular we do not assume it's continuous nor need it be defined on all of $\cantor$. Also, note that the expression $$f``[x(n) =1] \subseteq \bigcup_{m \in X \setminus n} [x(m)=1]$$ means that if $n \in Y$ then there is an $m \geq n$ so that $m \in X \cap f(Y)$. 

\begin{proof}
Let $X_\infty = \{n \in X \; | \; f``[n =1] \subseteq \bigcup_{m \in X \setminus n} [m=1]\}$. By assumption this is an infinite set. We want to show that it is in fact a pseudointersection to $\mathcal T$. Suppose not and fix $\alpha < \omega_1$ so that $X_\infty \nsubseteq^* X_\alpha$. Let $Y \in B$ be such that $f(\omega \setminus X_\alpha) = Y$. Since $Y$ is almost disjoint from $X$ there is a $k < \omega$ so that $Y \cap X \subseteq k$. Let $n \in X_\infty \setminus X_\alpha$ be larger than $k$. Since $n \notin X_\alpha$ we have that $n \in \omega \setminus X_\alpha$ and hence there is an $m\geq n$ in $X$ so that $m \in Y =  f(\omega\setminus X_\alpha)$. But this contradicts the choice of $k$ and $n$. 

\end{proof}

We're now ready to prove Theorem \ref{killing towers}. The idea is to take $A, B \subseteq 2^\omega$ which are $\aleph_1$-dense in $2^\omega$ as defined in Lemma \ref{little} and add an isomorphism from the image of $A$ and to the image of $B$ under the Cantor-Lebesgue map and then argue that this suffices. The following numerical fact will be useful and follows readily from what was discussed about $\lambda$ in the preliminaries. For each $n < \omega$ let us call a closed diadic interval {\em good at} $n$ if it is of the form $[\frac{2i+1}{2^{n+1}}, \frac{2i + 2}{2^{n+1}}]$. In other words, the set $\lambda``[x(n) = 1]$ is simply the union of the $2^{n}$ many good at $n$ intervals. 

\begin{proposition}
    Let $\epsilon \in (0, 1)$ and let $I \subseteq (0, 1)$ be an open interval of length $\epsilon$. If $n$ is large enough that $\frac{3}{2^{n+1}} < \epsilon$ then $I$ contains a closed interval good at $n$. \label{numbers}
\end{proposition}


\begin{proof}
    Fix $I$, $\epsilon$ and $n$ as in the statement of the proposition. Since we can cover $[0, 1]$ by the intervals $[\frac{i}{2^{n+1}}, \frac{i+1}{2^{n+1}}]$ some convex subset of these will cover $I$. Thus let us pick $0\leq i < j \leq 2^{n+1}$ so that $I \subseteq [\frac{i}{2^{n+1}}, \frac{j}{2^{n+1}}]$ with $i$ largest and $j$ smallest with this property. Note that this minimality then guarantees that $I \supseteq [\frac{i+1}{2^{n+1}}, \frac{j-1}{2^{n+1}}]$. By the assumption on $\epsilon$ we must have that $j -i > 3$. But this means that there is a $i< k < j$ so that $[\frac{k}{2^{n+1}}, \frac{k + 1}{2^{n+1}}]$ and $[\frac{k+1}{2^{n+1}}, \frac{k + 2}{2^{n+1}}]$ are both contained in $I$ and one of these is good at $n$ (depending on if $k$ is odd or even). 
\end{proof}

We are now ready to prove the main lemma towards the proof of Theorem \ref{LOmain}. Before stating it, we specify some terminology. Suppose $A, B \subseteq (0, 1)$ and $\P$ is reasonable for $A$ and $B$ (note that the difference between $(0, 1)$ and $\mathbb R$ is immaterial but will allow us to more easily apply the Cantor-Lebesgue map). We will denote by $\dot{g}_{A, B}$ the canonical name for the generic isomorphism from $A$ to $B$. When we work in the extension then this will be called $g_{A, B}$. Since this isomorphism determines uniquely an isomorphism from $[0, 1]$ to $[0, 1]$ we will confuse the two. 

\begin{lemma} \label{mainlemma}
   Suppose $X \in [\omega]^\omega$. Let $A, B \subseteq (0, 1)$ be $\aleph_1$-dense sets containing no diadic rationals so that the following holds:
   \begin{enumerate}
       \item For every finitely many $a_0, ..., a_{n-1} \in A$ there are infinitely many $l \in X$ so that no $a_i$ is in an interval which is good for $l$. In other words there is an infinite set of $l \in X$ so that $a_0, ..., a_{n-1} \in \lambda``[x(l) = 0]$. \label{assumption on A}
   \end{enumerate}

   Then, if $\P$ is reasonable for $A$ and $B$, then $\P$ forces that the generic isomorphism $\dot{g}_{A, B}$ will have the property that there are infinitely many $n \in X$ so that for every closed interval $I$ which is good for $n$, there is an $m \in X \setminus n$ so that $\dot{g}_{A, B}$ maps $I$ into some interval good for $m$. In other words, there are infinitely many $n \in X$ so that $\dot{g}_{A, B} ``(\lambda``[x(n) = 1]) \subseteq \bigcup_{m \in X \setminus n} \lambda ``[x(m) = 1]$. \label{control}
\end{lemma}

In the proof it will be useful to order (closed or open) intervals $I, J \subseteq (0, 1)$ as $I < J$ if and only if each $a \in I$ is less than every $b \in J$. Similarly we will write $a < I < b$ for elements $a < b \in (0, 1)$ if $a$ is less than the left endpoint of $I$ and $b$ is greater than the right endpoint of $I$. 

\begin{proof}
    Fix $A, B, X$ and $\P$ as in the statement of the lemma. Clearly it suffices to show that if $p \in \P$, $k \in \omega$ there the set of $q \leq p$ which force some $n \in X \setminus k$ to be such that $\dot{g}_{A, B} ``(\lambda``[x(n) = 1]) \subseteq \bigcup_{m \in X \setminus n} \lambda ``[x(m) = 1]$ is dense. Fix $p$. The domain of $p$ is some set (enumerated in order) $x_1 < ... < x_m$. Similarly the range is some set $z_1 < ... < z_m$ and, since it is an order isomorphism, it must be the case that for each $i \in [1, m]$ we have $p(x_i) = z_i$. The indexing begins at 1 to accommodate the following convenient notation. Let us treat $p$ as having also $0$ and $1$ in the domain and range with $p(0) = 0$ and $p(1) = 1$. Let $x_0 = 0 = z_0$ and $x_{m+1} = 1 = z_{m+1}$. 

    By (\ref{assumption on A}) in the statement of the lemma there is some $n \in X \setminus k$ so that no $x_i$ (other than $x_{m+1} = 1$ of course) is in an interval which is good for $n$. Enumerate in order the $2^{n}$ many closed intervals $t_0 < t_1 < ... < t_{2^{n}-1}$ which are good for $n$. Note that for no $j< m$ and is it the case that $x_j \in t_i$ for any $i < 2^n$. Thus, for each $i < 2^{n}$ there are $x_{j_i} < x_{j_i+1} \in {\rm dom}(p)$ so that $t_i \subseteq (x_{j_i}, x_{j_i + 1})$. Note that $x_{j_i}$ and $x_{j_{i}+1}$ may be the same for different $j$'s if e.g. the domain of $p$ misses multiple consecutive good for $n$ intervals. By the density of $A$, for each $i < 2^n - 1$ we can choose $a_0^i < a_1^i \notin {\rm dom}(p)$ so that the following hold:
    \begin{enumerate}
        \item $x_{j_i}< a^i_0 < t_i <  a^i_1 < x_{j_{i}+1}$
        \item $a^i_0 < t_i < a^i_1 < a^{i+1}_0 < t_{i+1} < a^{i+1}_1$
    \end{enumerate}
    The final good interval, $t_{2^n -1} : = [\frac{2^{n+1}-1}{2^n+1} , 1]$ requires slightly more care in notation but not conceptually. Here pick simply an $a^{2^n-1}_0 \in A$ greater than $x_n$ but less than $\frac{2^{n+1}-1}{2^n+1}$ and let $a^{2^n-1}_1 = 1$. In English, this means roughly in total that we choose a small open interval around each $t_i$ with endpoints in $A \setminus {\rm dom}(p)$ respecting the order on the $t_i$'s and refining the intervals $(x_{j_i}, x_{j_i + 1})$. 

\begin{claim}
    For each $i < m$, and each natural number $j > 0$ there is an $l \in X \setminus n$ and $j$ many distinct intervals $\{s^i_k \mid k < j\}$ which are each good for $l$ so that $s_i \subsetneq (z_{i}, z_{{i}+1})$.
\end{claim}

\begin{proof}
    Let $\epsilon < \frac{z_{i+1} - z_{i}}{j}$ and let $l \in X\setminus n$ be large enough that $\frac{3}{2^{l+1}} < \epsilon$. The claim now follows immediately by proposition \ref{numbers} by applying it $j$ many times to each of the intervals given by partitioning $(z_{i}, z_{i + 1})$ into $j$ many equal pieces. 
\end{proof}
For each $i < m$ let $j_i$ be the number of intervals of the form $t_j$ included in $(x_{i}, x_{i + 1})$. Applying the claim to $i$ and $j_i$ we obtain an $l_i \in X \setminus n$ and $j_i$ many intervals which are good for $l_i$, enumerated in order $s^i_{k} \subsetneq (z_{i}, z_{{i}+1})$. Moreover, since we applied the claim finitely many times and the proof of it in fact worked on a tail of $l\in X$ we can assume that $l_i$ is the same $l \in X \setminus n$ for all $i$. This is not strictly speaking necessary for the proof, or the application we intend of this lemma, however it is convenient insofar as it greatly simplifies notation so we use it. Thus fix a single $l$ that works as above as well as the intervals $s_i$. 

Putting all of these objects together we get a sequence (in order) of $2^{n}$ many intervals $s_i$ which are good for $l$ and sit in the relation to the range of $p$ in exactly the same way as the $t_i$'s sit in relation to the domain of $p$. We can therefore apply the dense mapping property $2 \times 2^n$ many times (by a finite induction) to find a condition $q \leq p$ so that for all $i < 2^n$ we have $q(a^i_0), q(a^i_1) \in s_i$ (and these numbers are defined). This $q$ is in fact the condition we need which completes the proof. 

This is in fact immediate from the definitions but we acknowledge there are a lot of definitions. Thus, in an effort to make the proof more accessible to a reader lost in the sea of indices let us illustrate this last point with a concrete example. Suppose there is a single $t_i$ contained between $x_0$ and $x_1$. We found $a^i_0, a^i_1\in A$ so that $x_0 < a^i_0 < t_i < a^i_1 < x_1$ held. Moreover we found some $l > n$ in $X$ and an interval $s_i$ which is good for $l$ in $(z_0, z_1)$. The condition $q$ now maps $a^i_0, a^i_1$ into the interval $s_i$ - which is legal because it respects the order of the elements of $A$ - so we get $(q(a^i_0), q(a^i_1)) \subsetneq s_i$. But then, since we are adding a linear order isomorphism $q$ must force that the image of $t_i$ under the generic isomorphism is included in $s_i$. Since we do this simultaneously for all $i$ we obtain that $q \forces \dot{g}_{A, B}``\lambda [x(n) = 1] \subseteq \lambda ``[x(l)= 1]$. 
\end{proof}

Let us now put this all together to prove Theorem \ref{LOmain}.

\begin{proof}[Proof of Theorem \ref{LOmain}]
Let $\mathcal T$ be a tower of size $\aleph_1$, enumerated as, say $\{X_\alpha\; | \; \alpha \in \omega_1\}$ with $X = X_0$ coinfinte and $\mathcal T$ $\aleph_1$-dense in $2^\omega$. Let $A = \{\omega \setminus X_\alpha \mid \alpha \in \omega_1\}$ and let $B$ be some $\aleph_1$-dense set of $2^\omega$ consisting of only elements which are almost disjoint from $X$. Note that $\lambda``A$ and $\lambda``B$ are both $\aleph_1$-dense. Let $\P$ be reasonable for $\lambda``A$ and $\lambda``B$. Work in $V[g_{\lambda``A, \lambda``B}]$. For any finitely many elements of $\lambda``A$ there is an infinite subset of $X$ disjoint from all of these since $\mathcal T$ is a tower of almost subsets of $X$. Therefore we can apply Lemma \ref{control} to obtain that there are infinitely many $n \in X$ so that $g_{\lambda ``A, \lambda ``B} ``(\lambda``[x(n) = 1]) \subseteq \bigcup_{m \in X \setminus n} \lambda ``[x(m) = 1]$. Finally now consider the map $f:A \to B$ defined by $\lambda^{-1} \circ g_{\lambda``A, \lambda``B} \circ \lambda \hook A$. It follows that $f$ satisfies the hypotheses of Lemma \ref{little} so we are done\footnote{We don't need it but this is actually a continuous map on a dense $G_\delta$ subset of $2^\omega$ containing $A$.}. 
    
\end{proof}

An immediate corollary of this is the following.
\begin{corollary}
    If $\{\P_\alpha, \dot{\Q}_\alpha \mid \alpha < \delta\}$ is a countable support iteration of proper forcing notions or a finite support iteration of ccc forcing notions for that for all pairs $A, B \subseteq \mathbb R$ which are $\aleph_1$-dense in the extension there is an $\alpha$ so that $A, B \in V[G_\alpha]$ and  $\forces_\alpha$``$\dot{\Q}_\alpha$ is reasonable for $\check{A}$ and $\check{B}$" then $\forces_\delta$``$\mfp > \aleph_1$".
\end{corollary}

\section{Adding Pseudointersections by forcing $\BA (2^\omega)$}

It is interesting to ask whether Theorem \ref{killing towers} holds for forcing notions forcing $\BA (X)$ for other Polish spaces $X$. We do not have a general result of this form as we are not sure what the right notion of ``reasonable" should be here. However, a strong variation of Lemma \ref{mainlemma} holds for the natural forcing for $\BA (2^\omega)$. This is the forcing notion due to Medini from \cite[p. 139]{Medini15}. 

\begin{definition}
    Let $A$ and $B$ be $\aleph_1$-dense subsets of $\cantor$ and let $A = \bigcup_{\alpha\in \omega} A_\alpha$ be a partition of $A$ into countable dense subsets, similarly let $B = \bigcup_{\alpha \in \omega_1} B_\alpha$ be the same. The forcing notion $\mathbb M_{A, B}$ consists of pairs $p = (\pi_p, f_p)$ so that the following hold.
    \begin{enumerate}
        \item $f_p:A \to B$ is a finite partial injection.
        \item For all $\alpha < \omega_1$ and $a \in {\rm dom}(f_p)$ we have $a \in A_\alpha$ if and only if $f_p(a) \in B_\alpha$. 
        \item There is an $n = n_p \in \omega$ so that $\pi_p$ is a permutation of $2^n$. 
        \item $f_p$ {\em respects} $\pi_p$ i.e. for all $a \in {\rm dom}(f_p)$ we have $f_p(a) \hook n = \pi_p(a \hook n)$. 
    \end{enumerate}
    Given conditions $p, q \in \mathbb M_{A, B}$ define $p \leq q$ if and only if $f_p \supseteq f_q$ and for all $t \in 2^{n_p}$ we have $\pi_p(t) \hook n_q = \pi_q(t\hook n_q)$. 
\end{definition}

It is not hard to check this forcing notion adds a homeomorphism $h:\cantor \cong \cantor$ mapping $A$ to $B$. Moreover it is ccc, in fact $\sigma$-centered, see the proof of \cite[Theorem 2.1]{Medini15}. 

\begin{theorem}
    Let $\mathcal T$ be a tower of size $\aleph_1$, $X$ be the infinite co-infinite set at the top of the tower and let $A$ and $B$ be the $\aleph_1$-dense sets defined in Lemma \ref{little}. Then $\mathbb M_{A, B}$ forces that the generic homeomorphism $h:\cantor \to \cantor$ so that $h``A = B$, has the property that there are infinitely many $n\in X$ so that $h``[x(n)=1] = [x(n)=1]$. In particular, $\mathbb M_{A, B}$ adds a pseudo-intersection to the tower $\mathcal T$.
\end{theorem}

\begin{proof}
    This is similar, though easier, than what we have seen before so we sketch the idea and leave the details to the reader. Given a natural number $k$ and a condition $p = (f_p, \pi_p) \in \mathbb M_{A, B}$ we want to find a $q \leq p$ and an $n \in X \setminus k$ so that $q$ forces the generic homeomorphism to map $[x(n) = 1]$ to itself. Indeed, since the domain of $f_p$ contains finitely many points in the tower and the range is finitely many elements of $\cantor$ which are almost disjoint from $X$ it is easy to find infinitely many $n \in X$ so that $n \notin \bigcup_{a \in {\rm dom}(p)} a \cup \bigcup_{b \in {\rm range}(p)} b$. If $n$ is such a number which is moreover larger than $k$ and $n_p$ (the natural number so that $\pi_p$ has domain $2^{n_p}$) then clearly we can find a permutation $\pi'$ of $2^{n+1}$ extending $\pi_p$ and compatible with $f_p$ so that $\pi`$ maps every element of $2^{n+1}$ with $1$ as its final bit to another element of $2^{n+1}$ with $1$ as its final bit. Clearly $q = (f_p, \pi')$ is the desired condition. 
\end{proof}

A similar argument can be made for the original forcing for forcing $\BA (\cantor)$ from \cite{BaldwinBeaudoin89}. Since this forcing is more involved we omit the argument. Note that the condition $h``[x(n)=1] = [x(n)=1]$ is much stronger than what we get in the case of linear orders. It is also not possible for linear orders, at least not in the setup we had in Section 2, as the only way this could happen for a fixed $n$ is if the generic isomorphism was the identity on every diadic rational appearing at the endpoints of a good at $n$ interval. However if this occurred for infinitely many times then the generic isomorphism would be the identity on a dense set and hence the identity which is obviously impossible. 

\section{Conclusion}

It would be nice to be able to apply the proof strategy from Theorem \ref{killing towers} to obtain that $\BA$ implies $\mfp > \aleph_1$. A natural attempt would be to replace the finiteness of the conditions with the fact that every $\aleph_1$-dense set under $\BA$ (indeed under $\mfb > \aleph_1$) can be covered by countably many closed (even compact) nowhere dense sets and attempt to do some kind of back and forth argument. Part of this is easy as the following lemma shows. 

\begin{lemma} \label{key}
    Assume $\BA$. Let $A$ and $B$ be $\aleph_1$-dense sets of reals. Let $F \subseteq \mathbb R \setminus A$ and $G \subseteq \mathbb R \setminus B$ be closed, nowhere dense sets and let $\varphi:F \to G$ be an order isomorphism between them. There is an isomorphism $\Phi:A \cup F \to B \cup G$ extending $\varphi$.
\end{lemma}

\begin{proof}
    Fix $F$, $G$, $A$, $B$ and $\varphi$ as in the statement of the lemma. We call an open interval $(a, b) \subseteq \mathbb R \setminus F$ {\em maximal} if it is not contained in any properly larger open interval disjoint from $F$. Let us denote by $M(F)$ the set of maximal open intervals disjoint from $F$ and similarly for $M(G)$ etc. The following can be checked easily.

    \begin{fact}
        The following all hold for $F$ (and also, for $G$).
        \begin{enumerate}
            \item Every $a \in \mathbb R \setminus F$ is contained in a single $I \in M(F)$.
            \item Every $I \in M(F)$ has the form $(a, b)$ for some $a < b \in F$.
            \item If $I, J \in M(F)$ are distinct then they are disjoint.
        \end{enumerate}
    \end{fact}

    The third item immediately implies that if $I , J \in M(F)$ then we can order them $I < J$ if and only if every $x \in I$ is less than every $y \in J$ and this is a total order on $M(F)$ (or $M(G)$ respectively). We will always treat $M(F)$ as a linear order with this ordering moving forward. The point is the following.

    \begin{claim}
        $\varphi:F \to G$ induces an order isomorphism $\hat{\varphi}:M(F) \cong M(G)$.
    \end{claim}

    \begin{proof}[Proof of Claim]
        If $I \in M(F)$ then by item (2) above $I = (a, b)$ for some $a < b \in F$ so we let $\hat{\varphi}(I) = (\varphi(a), \varphi(b))$. First let us see that this is well defined. If $I \in M(F)$ is as above the $\varphi(a) < \varphi(b) \in J$ and, since $\varphi$ is an order isomorphism and $F \cap (a, b)= \emptyset$ we must have that $\hat{\varphi}(I) \in M(G)$. Injectivity is immediate so let's check surjectivity. If $J \in M(F)$ then $J = (c, d)$ for some $c < d \in G$ and by an argument symmetric to the one proving well definedness we get that $(\varphi^{-1}(c), \varphi^{-1}(d)) \in M(F)$ and $J$ is the image of this interval under $\hat{\varphi}$. Thus $\hat{\varphi}$ is a well defined bijection between $M(F)$ and $M(G)$. Finally, if now $I < J \in M(F)$ then we must have that $a < b \leq c < d$ where $I = (a, b)$ and $J = (c, d)$. Consequently $\hat{\varphi}(I) < \hat{\varphi}(J)$ since $\varphi$ is assumed to be order preserving and hence we have $\varphi(a) < \varphi(b) \leq \varphi(c) < \varphi(d)$ so $\hat{\varphi}$ is order preserving.
    \end{proof}

Now, for each $I$ in $M(F)$ by $\BA$ there is an order isomorphism $\psi_I: I \cap A \cong \hat{\varphi}(I) \cap B$. We fix one for each such $I$. Finally let $\Phi:A \cup F \to B \cup G$ be the map defined simply as $\bigcup_{I \in M(F)}\psi_I \cup \varphi$. Let us check that this is the required isomorphism. It is nearly immediate from what has come before that this is a bijection, extends $\varphi$ and maps $A$ onto $B$. Finally since it is a union of order isomorphisms mapping ordered intervals to one another in order it is again order preserving. 
\end{proof}

Therefore one can fix a predetermined isomorphism from some closed nowhere dense $F$ to another closed nowhere dense $G$ and extend it to an isomorphism from $A \setminus F$ to $B \setminus G$ for any $\aleph_1$-dense $A$ and $B$ (note that an $\aleph_1$-dense set minus a nowhere dense set is still $\aleph_1$-dense). The issue is that it is not clear in the above why, borrowing from the proof of Lemma \ref{mainlemma} we can find infinitely $n \in X$ not in any element of the closed nowhere dense set $F$. If this technical issue could be overcome it would be straightforward to apply what is done in this paper to answer Question \ref{pquestion} affirmatively. 

\bibliographystyle{plain}
\bibliography{BAforcing}

\end{document}